\documentclass[leqno]{article}%
\usepackage{amsfonts}
\usepackage{amssymb}
\usepackage{amsmath}
\usepackage{graphicx}%
\usepackage{xcolor}
\usepackage{comment}
\providecommand{\U}[1]{\protect\rule{.1in}{.1in}}
\newtheorem{theorem}{Theorem}

\newtheorem{corollary}[theorem]{Corollary}

\newtheorem{definition}[theorem]{Definition}

\newtheorem{proposition}[theorem]{Proposition}
\newtheorem{remark}[theorem]{Remark}

\newenvironment{proof}[1][Proof]{\noindent\textbf{#1.} }{\ \rule{0.5em}{0.5em}}

\def \con{\nabla}
\def \conv{\nabla_{\mathbf{X}_v}}

\def \Riem{\mathrm{Riem}}

\def \xu{\mathbf{X}_u}
\def \xv{\mathbf{X}_v}
\def \okap{\overline{\kappa}}

\begin{document}

\title{\textbf{Ruled surfaces in $3$-dimensional Riemannian manifolds}}

\author{\textsc{Marco Castrill\'{o}n L\'opez$^1$}
\and \textsc{M. Eugenia Rosado Mar\'{\i}a$^2$}
\and \textsc{Alberto Soria Marina}$^3$ \vspace{0.2cm} \\
$^1$ Departamento de\'Algebra, Geometr\'ia y Topolog\'ia, \\
Facultad de Matem\'aticas, Universidad Complutense de Madrid,  \\
Plaza de las Ciencias, 3, 28040 Madrid, Spain \\
mcastri@mat.ucm.es \vspace{0.2cm} \\
$^2$ Departamento de Matem\'atica Aplicada, \\
ETS de Arquitectura, Universidad Polit\'ecnica de Madrid, \\
Avenida de Juan de Herrera, 4, 28040 Madrid \\
eugenia.rosado@upm.es \vspace{0.2cm} \\
$^3$ Departamento de Matem\'atica Aplicada a las TIC, \\
ETSI de Telecomunicaci\'on, Universidad Polit\'ecnica de Madrid, \\
Avenida Complutense, 30, 28040 Madrid \\
alberto.soria@upm.es
}
\date{}
\maketitle

\begin{abstract}

In this work, ruled surfaces in $3$-dimensional Riemannian manifolds are studied. We determine the expression for the extrinsic and sectional curvature of a parametrized ruled surface, where the former one is shown to be non-positive. We also quantify the set of ruling vector fields along a given base curve which allows to define a relevant reference frame that we refer to as \emph{Sannia frame}. The fundamental theorem of existence and equivalence of Sannia-ruled surfaces in terms of a system of invariants is given. The second part of the article tackles the concept of the striction curve, which is proven to be
the set of points where the so-called \emph{Jacobi evolution function} vanishes on a ruled surface. This characterization of striction curves provides independent proof for their existence and uniqueness in space forms and disproves their existence or uniqueness in some other cases.

\end{abstract}

\bigskip


\smallskip
\smallskip
\smallskip


\smallskip


\section{Introduction\label{sec0}}

 Ruled surfaces constitute a distinguished family of submanifolds in the
area of Differential Geometry with a long history and a rich collection of works in the literature. Even though their original framework in the Euclidean space is still an active field of research (just to mention one, the reader may have a look at the study of non-developable ruled surfaces in \cite{LiuYuJung2014}), ruled surfaces are also a prominent object in other Riemannian manifolds. For example, in the case of the Heisenberg group, where
parts of planes, helicoids and hyperbolic paraboloids are proven to be the only minimal ruled surfaces  (see \cite{Shinetal2013}), or the minimal (resp. maximal) ruled surfaces in the Bianchi-Cartan-Vranceanu spaces $\mathbb{E}(\kappa,\tau)$ (resp. in their Lorentzian counterparts), see
\cite{AlbujerdosSantos2021}. In \cite{DaSilvaDaSilva2022} several results for ruled surfaces in $3$-dimensional Riemannian manifolds are established, where formulas for the striction curve, the distribution parameter and the first and second fundamental forms of ruled surfaces in space forms are obtained. The authors also find a model-independent proof for the well-known fact that surfaces in $3$-dimensional space forms with vanishing (extrinsic) Gauss curvature are necessarily ruled. In addition, a necessary and sufficient condition on the ambient Riemann tensor along  any extrinsically flat surface in generic $3$-dimensional Riemannian manifolds is derived and required for the surface to be ruled. Specifically, if $Z$ and $X$ are orthogonal tangent vector fields to the surface where $Z$ is tangent to the rulings and $N$ stands for the normal vector field, such requirement is given by $\mathrm{Riem}(X,Z,Z,N)=0$.
Ruled surfaces have also been studied in Lorentzian contexts as in
\cite{DillenKuhnel1999}, where all Weingarten surfaces are characterized in the $3$-dimensional Minkowski space. Other results concerning ruled surfaces in the $3$-dimensional Minkowski space have been established in \cite{Choi1995,KimYoon2000,KimYoon2004}. Lightlike ruled hypersurfaces have also been of great utility to address the problem of some geometric inequalities as in \cite{MarsSoria2016}, where some cases of the so-called null Penrose inequality for the Bondi energy of a spacetime are proven.

All Euclidean ruled surfaces except cylinders admit a unique special curve called \emph{striction curve}, which is made up of the so-called central points of each generator \cite{PottmannWallner}. In the case of the cylinder, the uniqueness property is not fulfilled, for every curve in it is a striction one. On the other hand, striction curves can degenerate to a single point, as occurs with the cone, where it is made up exclusively of its vertex. A central point in a geodesic generator of the surface furnish a critical point to the distance from points on that generator to neighboring geodesics.
The family of generators along the striction curve is geodesically parallel along it. Striction curves are important, among other reasons, because they contain the singular points of the ruled surface
in case they exist \cite{DoCarmo2016}. To the best of our knowledge there are no results in the literature involving a method to determine the striction curve of ruled surfaces in generic Riemannian manifolds. In this work we present a new strategy to determine the existence of striction curves in such context based on a function referred to as \emph{Jacobi evolution function}, which vanishes at these type of curves in case they lie on the surface. Even though the presence of striction curves is addressed in \cite{DaSilvaDaSilva2022} for ruled surfaces in $S^3(r)$ and $\mathbb{H}^3(r)$, no results are found in the literature where either the existence or uniqueness of the striction curve is violated in such contexts. In fact, in this work, the aforementioned {\it Jacobi evolution function} is applied to prove that there exists at most one striction curve on ruled surfaces in the $3$-dimensional hyperbolic space $\mathbb{H}^3$, with examples of surfaces with no striction curve.

In case that the striction curve of a ruled surface in a $3$-dimensional Riemannian manifold has a regular striction curve, it is possible to define a specific reference frame called  \textit{Sannia} after the Italian mathematician who first proposed it (see \cite {PottmannWallner}, \cite {1925Sannia}), and which is characterized for having the generating vector of the rulings as the first element of the basis. The evolution of the  \textit{Sannia frame} provides two Euclidean invariants associated with the surface referred to as \textit{curvature} and \textit{torsion}, which together with the (striction) angle enclosed, determine a complete system of Euclidean invariants for the ruled surface.
A fundamental problem in Riemannian Geometry is that of equivalence of objects in a determined class,
namely to provide a criterion to know whether two given objects in this class are congruent under isometries or not. In \cite{Castrillonetal2015} the problem of curves with values in an arbitrary dimensional Riemannian manifold is solved with respect to the Frenet curvatures. However, spaces with non-constant curvature require additional invariants to establish such result. The classical Sannia  problem of reconstructing suitable surfaces provided a set of associated invariant
functions is known and is presented in Theorem 5.3.8 in \cite{PottmannWallner}. In this work we establish an analogous
result of local existence and uniqueness of ruled surfaces in generic $3$-dimensional Riemannian manifolds. In our main Theorem we prove that four functions of suitable regularity and the Sannia frame
at a given point $p_0\in M$ determine uniquely a (local) ruled surface
passing through $p_0$. Moreover, we do not require the base to be a striction curve. This is a great advantage since, as we shall see, certain ruled surfaces lack this sort of curves.

The paper is structured as follows. In Section \ref{Ruledsurfaces} we introduce the concept of {\it parametrized ruled surface}, which is the most practical way to present ruled surfaces for the purposes of this work. As a first result, we obtain the expression for the sectional curvature of any ruled surfaces in a $3$-dimensional Riemannian manifold in terms of the ambient geometry and its second fundamental form. We also define the {\it extended distribution parameter function} that generalizes the distribution parameter for ruled surface in the Euclidean space. At the end of the section we introduce the concept of {\it Sannia ruled surface}, which we define as the ones admitting a ruling vector field with linearly independent covariant derivative. In particular this means that a Sannia frame can be defined along the base curve. An idea of the size of the set of vectors satisfying such a property is put forward in Proposition \ref{Propjetgeneralposition}. Deriving the evolution equations for the vectors of a Sannia frame along a given base curve gives rise to a set of two invariant functions associated with the {\it Sannia ruled surface} which, together with two additional angle functions, provide the fundamental theorem of ruled surfaces (see Theorem \ref{EDOruledtheorem} below). The existence tackled in the Theorem is completed with the uniqueness only in space forms. In Section \ref{Strictioncurves} we focus our interest on the existence and uniqueness of striction curves. We define a new function on the ruled surface that we refer to as {\it Jacobi evolution function}, which vanishes on the striction curves in case of existence. This makes such a function a useful tool for proving the presence of striction curves on ruled surfaces. In the case of constant curvature $k$, we compute its expression explicitly in the three possible cases of sign of $k$. In particular, for $k<0$, the striction curve may not exist (in contrast to \cite{DaSilvaDaSilva2022}). Furthermore, some extra examples of rules surfaces in product manifolds are put forward to disprove the uniqueness of the striction curve in such backgrounds. In the text, Einstein's summation convention will be assumed.

\section{Ruled surfaces in $3$-dimensional Riemannian manifolds}
\label{Ruledsurfaces}

\subsection{Definitions and properties}

Let $(M,g)$ be a $3$-dimensional Riemannian manifold and let $\psi \colon\Sigma\rightarrow (M,g)$ be an immersed surface in $(M,g)$. We recall some basic concepts of the geometry of hypersurfaces (see for instance
\cite{KN}). We denote by $\nabla$ and $\nabla ^\Sigma$ the Levi-Civita connections of $(M,g)$ and
$(\Sigma, \psi ^* g=g_\Sigma)$ respectively. Locally, we can assume $\Sigma$ to be embedded
in $(M,g)$. Let us choose a unit normal vector $\xi$ in a neighborhood $U$ of a
point $p\in\Sigma$. Recall that the second fundamental form of $\Sigma$
is defined as
\begin{align*}
	\vec{h}\colon T_p\Sigma\times T_p\Sigma &  \rightarrow
    T^\perp _p \Sigma
\\
	(U,V) &   \mapsto (\nabla_{U}V)_{p}^{\perp}=h(U,V)\xi,
\end{align*}
where $T^\perp _p \Sigma$ is the subspace of normal vectors on $\Sigma$, and $h\colon T_{p}\Sigma\times T_{p}\Sigma\rightarrow\mathbb{R}$ is
the associated second fundamental form tensor of $\Sigma$ defined as
\[
h(U,V)=-g(\nabla_{U}\xi,V).
\]

Given any vectors fields $X,Y,Z,T\in \mathfrak{X}(M)$, we consider
the curvature tensor of $(M,g)$ as
\begin{equation*}
	R(X,Y)Z=\nabla_X \nabla_Y Z-\nabla_Y \nabla_X Z-\nabla_{[X,Y]} Z,
\end{equation*}
and the Riemann tensor is  $\mathrm{Riem}(X,Y,Z,T)=-g(R(X,Y)Z,T)$. The induced curvature tensor and Riemann tensor on $\Sigma$ will be denoted by $R^\Sigma$
and $\mathrm{Riem}^\Sigma$ respectively.
Given any $p\in M$ and two linearly independent vectors $X,Y\in T_p M$ generating a plane $T_p \Sigma\subset T_p M$, the sectional curvature $K(T_p\Sigma)$ in $(M,g)$ is defined by
\begin{equation*}
	K(T_p\Sigma)=\frac{\mathrm{Riem}(X,Y,X,Y)}{g(X,X)g(Y,Y)-g(X,Y)^2},
\end{equation*}
and the sectional curvature of $\Sigma$ will be denoted by $K^\Sigma$. The sectional curvatures $K(T_p\Sigma)$ in $(M,g)$ and $K^\Sigma _p$ in $(\Sigma,g_\Sigma)$ are related by
\begin{equation}
K^\Sigma _p = K(T_p\Sigma) + K^\Sigma _{ext},
\label{ExtCurvatures}
\end{equation}
where $K^\Sigma_{ext}$ is the extrinsic or Gauss curvature (the determinant of the second fundamental form endomorphism).

Ruled surfaces are a remarkable family of surfaces in $3$-dimen\-sional Riemannian manifolds. They are generally defined as follows:
\begin{definition}
	We say that the immersed surface $\psi\colon \Sigma\rightarrow (M,g)$ in $(M,g)$
	is ruled if there exists a foliation of $\Sigma$ by curves which are geodesics in the ambient
	space $(M,g)$.
\end{definition}

However, in the following we are going to work with a notion, closer to the classical constructions in the Euclidean space, of ruled surface defined by base curves and ruling directions.

\begin{definition}
	Let $(M,g)$ be a complete $3$-dimensional Riemannian manifold. Let $\alpha\colon I\rightarrow (M,g)$ be a smooth regular curve and $Z$ a non-vanishing smooth vector field along $\alpha$, i.e. a curve in $TM$ such that $Z(u)\in T_{\alpha(u)}M$, $\forall u\in I$. The \emph{parametrized ruled surface} defined by $\alpha$ and $Z$ is the differentiable map
			\begin{align}
			\label{rulparam}
			\mathbf{X}\colon I\times\mathbb{R}  & \rightarrow (M,g) \nonumber \\
			(u,v)  & \mapsto\gamma_{Z(u)}(v)=\exp_{\alpha(u)}(vZ(u)),
		\end{align}
	where $\gamma_{Z(u)}\colon\mathbb{R}\rightarrow (M,g)$ is the geodesic satisfying
		\begin{align*}
			\gamma_{Z(u)}(0)  & =\alpha(u),\\
			\gamma_{Z(u)}^{\prime}(0)  & =Z(u).
		\end{align*}
\end{definition}

\begin{remark}
	    Obviously, the image of the map  $\mathbf{X}$ is not necessarily
		an embedded surface in $(M,g)$. But if we require the rank of $d\mathbf{X}$ to be 2, it will be an immersed surface.
	 On the other hand, if $(M,g)$ is not complete, the definition above is still valid by restricting the domain of $\mathbf{X}$.
\end{remark}

Note that the partial derivative $\xv=\partial \mathbf{X}/\partial v$ is a geodesic vector field, that is
\[
\nabla_{\xv}\xv=0.
\]
In addition, $\xv (u,0)=Z(u)$. The other partial derivative $\xu= \partial \mathbf{X}/\partial u$ is a Jacobi vector field along the geodesic $\gamma_{Z(u)}$ since it is a geodesic variation. Hence
it satisfies the Jacobi equation
\begin{equation}
	\label{Jacobeq}
	\conv \conv \xu+R(\xu,\xv)\xv=0.
\end{equation}

In the following result we derive the expression for the (intrinsic) curvature $K_p^\Sigma$ of a ruled surface $\Sigma$ in $(M,g)$.
\begin{proposition}
Let $\Sigma$ be a parametrized ruled surface in $(M,g)$ defined by
$\mathbf{X}\colon I\times\mathbb{R}\rightarrow (M,g)$, where $d\mathbf{X}$ is of rank $2$. Then, we have
\begin{equation}
	\label{intcurvature}
	K_p^{\Sigma}=\frac{-g(\nabla_{\xv} \nabla_{\xv} \xu,\xu)}{\vert\vert \mathbf{X}_u\vert\vert^2 \vert\vert \mathbf{X}_v\vert\vert^2-g(\mathbf{X}_u,\mathbf{X}_v)^2}-\frac{\mathrm{vol}_g(\xu,\xv,\nabla_{\xu}\xv)^2}{(\vert\vert \mathbf{X}_u\vert\vert^2 \vert\vert \mathbf{X}_v\vert\vert^2-g(\mathbf{X}_u,\mathbf{X}_v)^2)^2}.
\end{equation}

\end{proposition}

\begin{proof}
Taking into account the Gauss equation (see \cite[Volume II, Chapter
	VII, Proposition 4.1]{KN}), the sectional curvature for $T_{p}\Sigma$ is
	written as follows
	\begin{align*}
		K_p^\Sigma &  =\frac{\mathrm{Riem}^\Sigma(\mathbf{X}_u,\mathbf{X}_v,\mathbf{X}_u,\mathbf{X}_v)}{\vert\vert \mathbf{X}_u\vert\vert^2 \vert\vert \mathbf{X}_v\vert\vert^2-g(\mathbf{X}_u,\mathbf{X}_v)^2}\\
		&  =\frac{\mathrm{Riem}(\mathbf{X}_u,\mathbf{X}_v,\mathbf{X}_u,\mathbf{X}_v)+g(\vec{h}(\mathbf{X}_u,\mathbf{X}_u),\vec{h}
			(\mathbf{X}_v,\mathbf{X}_v))-g(\vec{h}(\mathbf{X}_u,\mathbf{X}_v),\vec{h}(\mathbf{X}_u,\mathbf{X}_v))}{\vert\vert \mathbf{X}_u\vert\vert^2 \vert\vert \mathbf{X}_v\vert\vert^2-g(\mathbf{X}_u,\mathbf{X}_v)^2}.
			\end{align*}
Since $\mathbf{X}_v$ is geodesic and $[\mathbf{X}_u,\mathbf{X}_v]=0$, we have
	\begin{equation*}
	  K_p^\Sigma=\frac{-g(\nabla _{\mathbf{X}_v}\nabla _{\mathbf{X}_v}\mathbf{X}_u,\mathbf{X}_u)-h(\mathbf{X}_u,\mathbf{X}_v)^2}{\vert\vert \mathbf{X}_u\vert\vert^2 \vert\vert \mathbf{X}_v\vert\vert^2-g(\mathbf{X}_u,\mathbf{X}_v)^2},	
 	\label{seccurvatureformula}
   \end{equation*}
	The proof is complete by taking into account
 \[
 |h(\mathbf{X}_u,\mathbf{X}_v)|=|g(\xi ,\nabla _{\mathbf{X}_u}\mathbf{X}_v)|=\frac{|\mathrm{vol}_g(\mathbf{X}_u,\mathbf{X}_v,\nabla _{\mathbf{X}_u}\mathbf{X}_v)|}{\sqrt{\vert\vert \mathbf{X}_u\vert\vert^2 \vert\vert \mathbf{X}_v\vert\vert^2-g(\mathbf{X}_u,\mathbf{X}_v)^2}}.
 \]

\end{proof}

\begin{remark}
	
Taking (\ref{ExtCurvatures}) and (\ref{intcurvature}) into account, it follows that
\begin{eqnarray}
	K(T_{p}\Sigma)
	&=&\frac{-g(\nabla_{\xv} \nabla_{\xv} \xu,\xu)}{\vert\vert \mathbf{X}_u\vert\vert^2 \vert\vert \mathbf{X}_v\vert\vert^2-g(\mathbf{X}_u,\mathbf{X}_v)^2}, \nonumber \\
	K_{\mathrm{ext}}^{\Sigma}
	&=&-\frac{\mathrm{vol}_g(\xu,\xv,\nabla_{\xu}\xv)^2}{(\vert\vert \mathbf{X}_u\vert\vert^2 \vert\vert \mathbf{X}_v\vert\vert^2-g(\mathbf{X}_u,\mathbf{X}_v)^2)^2}. \label{curvatures}
\end{eqnarray}
Note that the extrinsic curvature $K^\Sigma_{ext}$ of a parametrized ruled surface in a generic Riemannian background is non-positive, as (\ref{curvatures}) shows.
In particular, the (intrinsic) curvature of a ruled surface is always less or equal than the ambient sectional curvature. Also notice that the extrinsic curvature relation $K_{\mathrm{ext}}^{\Sigma}$ in (\ref{curvatures}) is still valid for any other basis of $T_p\Sigma$ since it is a tensorial expression.
\end{remark}

The expression of the extrinsic curvature that we have obtained above is connected with a classical invariant in the theory of ruled surfaces, the distribution parameter $\lambda$ of ruled surfaces in the Euclidean space (see \cite{DoCarmo2016,Hicks1965,PottmannWallner,Struik} for more details).
We next define a function defined on a parametrized ruled surface
motivated by the classical concept of distribution parameter:

\begin{definition}[\bf Extended distribution parameter]
\label{Defdistribparam}	
Let $\mathbf{X}\colon I\times\mathbb{R}\rightarrow (M,g)$ be a parametrized ruled surface as in (\ref{rulparam}) in a Riemmanian $3$-manifold $(M,g)$ with $d\mathbf{X}$ of rank $2$. We define the \emph{extended distribution parameter function} of $\mathbf{X}\colon I\times\mathbb{R}\rightarrow (M,g)$ as
\begin{equation}
	    \label{distribparamfunction}
		\lambda(u,v)=\frac{\mathrm{vol}_g(\xu,\xv,\nabla_{\xu}\xv)}{\Vert\nabla_{\xu}\xv\Vert^2}.
\end{equation}	
\end{definition}
\begin{remark}
As we will illustrate in Section \ref{Strictioncurves}, for $v=0$ the above formula reduces to the classical distribution parameter $\lambda(u,0)=\lambda(u)$ provided $\alpha\colon I\rightarrow \mathbf{X}(I\times\mathbb{R})$ is a striction base curve.
\end{remark}

\begin{definition}
	    \label{anglefunction}
		Let $\mathbf{X}\colon I\times\mathbb{R}\rightarrow (M,g)$ a parametrized ruled surface in $(M,g)$ as in (\ref{rulparam}).
		The function $\sigma\colon I\rightarrow\mathbb{R}$
		satisfying
	\[
	\cos\sigma_{\alpha(u)}=\frac{g(\alpha^{\prime}(u),Z(u))}{\Vert\alpha^\prime(u)\Vert },
	\]
	is called {\rm base angle} along the curve $\alpha$.
\end{definition}

It is a well-known fact that given a one-parameter family of geodesics parame\-trized by arc-length on a Riemannian manifold the product of their velocity by the associated Jacobi vector field is constant along each one (see for instance \cite{Hicks1965}). For the sake of completeness of this work we include a proof of this result adapted to the setting we are considering, which in turn allows us to find an expression for the angle $\sigma$ between $\xu$ and $\xv$ at any point $q$ of a ruled surface in terms of the geometry of its base curve and the norm of the Jacobi field $\xu$ at $q$:

\begin{proposition}
	\label{Propcoordframeangle}
	Let $\mathbf{X}\colon I\times\mathbb{R}\rightarrow (M,g)$ be a parametrized ruled surface as in (\ref{rulparam}), being $Z$ a unit vector field. Then $g(\xu,\xv)$ is constant along its rulings. Moreover, the angle $\sigma$ between $\xu$ and $\xv$ at any point $q$ of the ruling containing $p=\alpha(u)$ is given by
	\begin{equation}
		\label{coordframeangle}
		\cos{\sigma}_q=\frac{\Vert\alpha^\prime(u)\Vert\cos{\sigma_p}}{\Vert\xu\Vert_q}.
	\end{equation}
\end{proposition}
\begin{proof}
	Differentiating the function $g(\xu,\xv)$ along the vector field $\xv$ gives
	\begin{equation*}
		\xv(g(\xu,\xv))=g(\conv \xu,\xv)+g(\xu,\conv \xv)=g(\nabla_{\xu} \xv,\xv)=0,
	\end{equation*}
	where we have taken into account that $[\xu,\xv]=0$ and $\xv$ is geodesic. This means that $g(\xu,\xv)=\Vert\xu\Vert\cos{\sigma}$ is constant along the rulings.
	This value can be obtained by evaluating it at the point $p=\alpha(u)$. Indeed,
	given any $q$ of the ruling at $p=\alpha(u)$,
	\begin{equation*}
		g(\xu,\xv)|_q=\Vert\xu\Vert_q\cos{\sigma_q}=\Vert\xu\Vert_p\cos{\sigma}_p=\Vert\alpha^\prime(u)\Vert\cos{\sigma_p},
	\end{equation*}
	from where relation (\ref{coordframeangle}) follows.
	\end{proof}

\begin{remark}
		Relation (\ref{coordframeangle}) shows that whenever the vectors  $\xu(u,0)=\alpha^\prime(u)$ and $Z_p=\xv(u,0)$ are orthogonal at the base curve, they remain orthogonal along the ruling $\gamma_{Z(u)}(v)=\mathbf{X}(u,v)$. Nevertheless, the coordinate basis $(\xu,\xv)$ associated to the parametrization (\ref{rulparam}) with $d\mathbf{X}$ of rank $2$ is not necessarily orthogonal.
\end{remark}

As already mentioned, $\xu$ is a Jacobi vector field on every parametrized ruled surface, not necessarily orthogonal to its rulings. However, it is sometimes useful  to consider the orthogonal component to the surface generators, which turns out to be a Jacobi field too. In the following Proposition we compute the decomposition of $\xu$ into its tangent and normal components to the rulings.

\begin{proposition}
	\label{Propjacdecomposition}
	Let $\mathbf{X}\colon I\times\mathbb{R}\rightarrow (M,g)$ be a parametrized ruled surface in $(M,g)$ as in (\ref{rulparam}), being $Z$ a unit vector field. Then the decomposition of the Jacobi field $\xu$ into its tangential and normal part with respect to the ruling is
	\begin{equation*}
		\label{Jacdecomposition}
		\xu=\Vert\alpha^\prime(u)\Vert(\cos{\sigma_p})\xv+\xu^{\bot},
	\end{equation*}
	where $\xu^{\bot}$ is a Jacobi field along the ruling $\gamma_{Z(u)}$ and orthogonal to it, and $\sigma_p$ is the
	base angle at $p=\alpha(u)$.
\end{proposition}
\begin{proof}
	It is a well known fact that any Jacobi vector field along a geodesic curve $\gamma(v)$
parametrized by its arc length decomposes as
	\begin{equation*}
		J(v)=(a+bv)\gamma'(v)+J^{\bot}(v),
	\end{equation*}
	where $a=g(J(0),\gamma'(0))$, $b=g((\nabla_{\gamma'} J)_p, \gamma'(0))$
and $J^{\bot}$  is an orthogonal Jacobi field along $\gamma$.
In this background, the Jacobi field $\xu$ evaluated at $p $ reads $\xu|_p=\alpha^\prime(u)$, so
	\begin{equation*}
		a=g(\alpha^\prime(u),Z_p)=\Vert\alpha^\prime(u)\Vert\cos{\sigma_p},
	\end{equation*}	
	and
	\begin{equation*}
		b=g((\nabla_{\xv} \xu)|_p,Z_p)=\frac{1}{2}\xu (g(\xv,\xv))|_p=0,
	\end{equation*}
	since $\Vert\xv\Vert=1$. Decomposition (\ref{Jacdecomposition}) follows from such values.
\end{proof}

\subsection{Sannia invariants and the Fundamental Theorem of Ruled Surfaces}

Orthonormal frames along curves are often considered in Geometry and Physics
to address problems in relation to the geometry of manifolds and submanifolds.
In this work we consider frames whose first vector is determined by the geodesics defining the ruled surface.
Under suitable regularity hypothesis, a possible way of constructing the rest of the vectors in the frame is by considering successive derivatives of the first one along the tangent direction to the curve (see for example \cite{Castrillonetal2015}). However, it may occur that some of the derivatives are linearly dependent with some other
vector in the frame, which would prevent such a set of vectors to become a basis. With the following results we first intend to give an idea of the size of the set of ruled surfaces admitting such a relevant frame, and then we prove the main result of this work, the \emph{fundamental theorem of ruled surfaces}, in which we state that certain ruled surfaces are uniquely determined from certain invariants.

\begin{definition}
	A vector field $Z\in\mathfrak{X}(\alpha)$ along a smooth curve $\alpha
	\colon I\rightarrow M$ taking values into a manifold $M$ endowed with a
	linear connection $\nabla$ is said to be in \emph{general position} at $u_{0}\in I$ if the vector fields $Z$ and $\nabla
	_{\alpha^{\prime}}Z$ along $\alpha$ are linearly independent at $u_{0}$. The
	vector field $Z\in\mathfrak{X}(\alpha)$ is in general position
	if it is in general position for every $u\in(a,b)$.
\end{definition}

The following result quantifies the set of vectors in general position along a curve on a manifold endowed with a linear connection.
To this purpose it will be necessary to make use of the so-called jet bundles. We refer the reader to
\cite[section 41]{KrieglMichor1997} for more details on this topic.
\begin{proposition}
	\label{Propjetgeneralposition}
	Let $M$ be a $3$-dimensional manifold endowed with a linear
	connection $\nabla$ and let $\alpha\colon I\rightarrow M$ be a smooth curve
	on $M$. The set of vector fields $Z\in\mathfrak{X}(\alpha)$ along $\alpha$
	that are in general position is a dense set on $\mathfrak{X}(\alpha)$ with respect to the strong topology.
\end{proposition}

\begin{proof}
	The sections of the bundle $E=\alpha^{\ast}TM\rightarrow I$ are vector
	fields along $\alpha$. Let consider the $1$-jet bundle of $E$. The morphism on
	$J^{1}E$ given by
	\begin{align*}
		\varrho\colon J^{1}E  &  \longrightarrow E\oplus E\\
		j_{u}^{1}Z  &  \mapsto\left(  Z(u),(\nabla_{\alpha^{\prime}}Z)(u)\right)
	\end{align*}
	is an isomorphism. We define the singular set
	\[
	S=\{(Z_{1},Z_{2})\in E\oplus E\colon Z_{1}\wedge Z_{2}=0\},
	\]
	which can written as the (non-disjoint) union $S=S_{1}\cup S_{2}$, with
	\begin{align*}
		S_{1}  &  =\{(Z_{1},Z_{2})\in E\oplus E\colon Z_{2}=fZ_{1}\}\simeq E\times
		\mathbb{R},\\
		S_{2}  &  =\{(Z_{1},0)\in E\oplus E\}\simeq E.
	\end{align*}
	Both $S_{1}$ and $S_{2}$ are closed submanifolds of $E\oplus E$, of dimensions
	$3+1$ (resp. codimension $2-1$) and $3$ (resp. codimension $3$) respectively.
	The same properties apply to $T_{1}=\varrho^{-1}(S_{1})$ and $T_{2}%
	=\varrho^{-1}(S_{2})$. According to Thom's transversality Theorem (\emph{cf.}
	\cite[VII, Th\'{e}or\`{e}me 4.2]{Tou}), the set of curves $Z(u)\in\Gamma(E)$
	such that $j^{1}Z$ is transversal to both $T_{1}$ and $T_{2}$ is open and
	dense in $\Gamma(E)$ with the strong topology. For these curves $Z(u)$, the
	codimension of $(j^{1}Z)^{-1}(T_{1})$ is $2$ and the codimension of
	$(j^{1}Z)^{-1}(T_{2})$ is $3$. Since these are set in $I\subset\mathbb{R}$,
	they must be empty. Therefore, for this dense set of curves $Z(u)$, we have
	that $j^{1}Z\cap S=\varnothing$, and the proof is complete.
\end{proof}

\begin{definition}
	Let $(M,g)$ be a Riemannian $3$-manifold. A parametrized ruled surface $\mathbf{X}\colon I\times\mathbb{R}\rightarrow (M,g)$ is said to be a \emph{Sannia ruled surface} if
	$Z\in\mathfrak{X}(\alpha)$ is in general position with respect to the
	Levi-Civita connection of $g$.
\end{definition}

\begin{proposition}
   \label{propSanniaderivatives}
	\label{PropSanniaFrame}Let $(M,g)$ be an oriented $3$-dimensional Riemannian manifold
	and let $\mathbf{X}\colon I\times\mathbb{R}\rightarrow (M,g)$ be a Sannia ruled surface
	defined by a vector field $Z$ along a smooth curve
	$\alpha\colon I\rightarrow (M,g)$. Then, there exist unique vector fields $X_{i}$, $1\leq i\leq3$, defined along $\alpha$ and smooth functions
	$\kappa_{i}\colon I\rightarrow\mathbb{R}$, $\,0\leq i\leq2$, with
	$\kappa_{0}>0$ and $\kappa_{1}>0$, such that
	
	\begin{enumerate}
		\item $(  X_{1}(u),X_{2}(u),X_{3}(u) ) $
		is a positively oriented orthonormal linear frame of $T_{\alpha (u)}M$ for $u\in I$.
		
		\item The following formulas hold:
		
		\begin{enumerate}
			\item $Z=\kappa_{0}X_{1}$, \label{Sannia0}
			
			\item $\nabla_{\alpha^{\prime}}X_{1}=\kappa_{1}X_{2}%
			,$  \label{Sannia1}
			
			\item $\nabla_{\alpha^{\prime}}X_{2}=-\kappa_{1}%
			X_{1}+\kappa_{2}X_{3},$ \label{Sannia2}
			
			\item $\nabla_{\alpha^{\prime}}X_{3}=-\kappa_{2}%
			X_{2}.$ \label{Sannia3}
		\end{enumerate}
	\end{enumerate}
\end{proposition}

\begin{proof}
	We define
	\begin{eqnarray*}
		X_{1}  & = &(\kappa_{0})^{-1}Z,\label{X1}\\
		X_{2}  & = &(\kappa_{1})^{-1}\nabla_{\alpha^{\prime}}		X_{1},\label{X2}%
	\end{eqnarray*}
	where $\kappa_{0}=\left\Vert Z \right\Vert $, $\kappa_{1}=\left\Vert \nabla_{\alpha^{\prime}}X_{1}\right\Vert $. For $X_3$ we consider the unique vector field defining a positive orthonormal basis with $X_1$ and $X_2$. Differentiating $g(X_i,X_j)$ with respect to $u$ we get the formulas  (\ref{Sannia0})-(\ref{Sannia3}) in the statement.
\end{proof}

\begin{definition}
\label{definv}
	The frame $(  X_{1},X_{2},X_{3} )  $
	along $\alpha$ determined in the above Proposition is called \emph{Sannia
		frame along} $\alpha$, and the functions $\kappa_{0},\kappa
	_{1},\kappa_{2}$ are the \emph{Sannia invariants} of the
	ruled surface $\mathbf{X}\colon I\times\mathbb{R}\rightarrow (M,g)$.

\end{definition}

Let $\theta$ and $\varphi$ denote the spherical angles of $\alpha '$ with respect to $(X_{1},X_{2},X_{3})$; that is,
 \begin{equation}
		\label{basecurvedecomposition}
		\frac{\alpha^\prime}{\Vert\alpha '\Vert}=\cos{\varphi}\cos{\theta}\, X_1+\sin{\varphi}\,X_2+\cos{\varphi}\sin{\theta}\,X_3.
	\end{equation}
 If $\varphi =0$, then $\theta$ is the base angle $\sigma$ (see Definition \ref{anglefunction}). We will show in Section \ref{Strictioncurves} that this case corresponds to the base curve being a striction one.

 \begin{remark}
It would be more accurate to collect the angle invariants $\theta$ and $\varphi$ into a single smooth function $\varsigma \colon I\to S^2\subset \mathbb{R}^3$ assigning the coordinates of $\alpha '$ with respect to the Sannia basis. However, in order to be close to the classical results in the Euclidean space, we will keep track of the angles $\theta$ and $\varphi$ instead of $\varsigma$.
\end{remark}

In Definition \ref{Defdistribparam} we have presented a function that extends the traditional Euclidean distribution parameter to the whole surface in generic $3$-dimensional Riemannian backgrounds.
With a view to recovering the distribution parameter in the classical sense,
we next give the value of this function in terms of the invariants $\kappa _0$, $\kappa _1$, $\kappa _2$, $\theta$ and $\varphi$ of the Sannia frame associated
with a base curve which does not have to be necessarily a striction curve.
\begin{proposition}
	If $( X_{1},X_{2},X_{3} )  $ is the Sannia frame of a Sannia ruled surface  $\mathbf{X}\colon I\times\mathbb{R}\rightarrow (M,g)$, then the value of the distribution parameter function $\lambda$
	on the base curve $\alpha$ is
	\[
 \lambda(u,0)=
 \frac{\Vert\alpha^\prime(u) \Vert\kappa_{0}^{2}(u)\kappa_1(u) \cos{\varphi}_{\alpha(u)}\sin{\theta}_{\alpha(u)}}{\left(\frac{d\kappa_0(u)}{du}\right)^2+\kappa_0^{2}(u)\kappa_1^{2}(u)}.
	\]
\end{proposition}

\begin{proof}
	The result follows directly when \eqref{basecurvedecomposition} and
	relations (\ref{Sannia0})-(\ref{Sannia3}) in Proposition \ref{propSanniaderivatives}
	are inserted into expression \eqref{distribparamfunction}, evaluated at $v=0$.
\end{proof}

\begin{remark}
	Note that it is always possible to consider a unit ruling $Z$ along an arc-length parametrized curve $\alpha$, i.e $\kappa_0=1$ and $\Vert \alpha^\prime \Vert=1$. In such case the value of the distribution parameter function along $\alpha$ reduces to
	\begin{equation}
		\label{basecurvdistribparam}		
		\lambda(u,0)=\frac{\cos{\varphi}_{\alpha(u)}\sin{\theta}_{\alpha(u)}}{\kappa_1(u)}.
	\end{equation}
\end{remark}

We next state and prove the local existence and uniqueness theorem of ruled surfaces in general $3$-dimensional Riemannian manifolds:

\begin{theorem}[Fundamental theorem of ruled surfaces]
	\label{EDOruledtheorem}
 	Let $(M,g)$ be a $3$-dimensional oriented Riemannian manifold and let
	$(e_{1},e_{2},e_{3})$ be a positively oriented orthonormal basis of
	$T_{p_{0}}M$, $p_0\in M$. Given smooth functions $\okap_{i}\colon(u_{0}-\varepsilon
	,u_{0}+\varepsilon)\rightarrow\mathbb{R}$, $0\leq i\leq 2
	$, $\overline{\theta} \colon (u_{0}-\varepsilon
	,u_{0}+\varepsilon)\to \mathbb{R}$, $\overline{\varphi}\colon (u_{0}-\varepsilon
	,u_{0}+\varepsilon)\to \mathbb{R}$ with $\overline{\kappa}_{0},\overline{\kappa}_1
	>0$, there exists $\delta<\varepsilon$, a curve
	$\alpha\colon(u_{0}-\delta,u_{0}+\delta)\rightarrow (M,g)$, parametrized by its arc length, and a vector field $Z\in \mathfrak{X}(\alpha)$ such that the Sannia ruled surface $\mathbf{X}\colon I\times\mathbb{R}\rightarrow (M,g)$, $\mathbf{X}(u,v)=\exp_{\alpha(u)}(vZ(u))$ satisfies:
	
	\begin{enumerate}
		\item $\alpha(u_{0})=p_{0}$,
		
		\item $X_{i}(u_{0})=e_{i}$ for $1\leq i\leq3$,
		
		\item $\kappa_{i}=\okap_{i}$ for $0\leq i\leq 2$, $\overline{\theta}= \theta$, $\overline{\varphi}=\varphi$,
	\end{enumerate}
where  $\kappa_{i}$, $0\leq i\leq2$, $\theta$ and $\varphi$ are the invariants of Definition \ref{definv}
and $X_{i}(u_{0})$ for $1\leq i\leq3$ is the Sannia frame.
In addition, given any two points $p_{0},p_{0}^{\prime}\in M$ and two oriented
orthonormal bases $(e_i)_{i=1}^3$, $(e'_i)_{i=1}^3$
of $T_{p_{0}}M$ and $T_{p_{0}^{\prime}}M$ respectively,
there always exists a local isometry around $p_{0}$ and $p_{0}^{\prime}$ sending one
ruled surface to the other if and only if $(M,g)$ is of constant curvature.
\end{theorem}

\begin{proof}
	We consider the vector bundle $p_{M}\colon\oplus^{3}TM\rightarrow M$ and a
	normal coordinate system $(U,(x_{i})_{i=1}^{3})$ centred at $p_{0}$ associated
	with the orthonormal basis $(e_i)_{i=1}^3$. We define the natural
	coordinate system $(x^{i},y_{k}^{j})$, $i,j,k=1,2,3$ in $p_{M}^{-1}(U)$ such
	that
	\[
	w_{j}=y_{j}^{i}(w)\left.  \frac{\partial}{\partial x^{i}}\right\vert
	_{p},\qquad\forall w=\left(  w_{1},w_{2},w_{3}\right)  \in\oplus^{3}%
	T_{x}M,\quad p\in U.
	\]
    Let $X\colon I \rightarrow\oplus^{3}TM$,
	$a<u_{0}<b$, be the curve given by
	\begin{equation*}
	X(u)=\left(  X_{1}(u),X_{2}(u),X_{3}(u)\right).	
	\end{equation*}
	    Without loss of generality, the ruling can be considered to be of unit length, i.e. $\okap_0=1$
    in Proposition \ref{PropSanniaFrame}. This condition together with the formulas  (\ref{Sannia0})-(\ref{Sannia3})     can be expressed as follows:
	\begin{equation}
		\left\{
		\begin{array}
			[c]{l}%
			\dfrac{d\left(  x^{i}\circ\alpha\right)  }{du}\!=\!
			\cos{\overline{\varphi}}\cos{\overline{\theta}}\,(y_1^i\circ X)+
			\sin{\overline{\varphi}}\,(y_2^i\circ X)+\cos{\overline{\varphi}}\sin{\overline{\theta}}\,(y_3^i\circ X),\smallskip\\
			\dfrac{d\left(  y_{1}^{i}\circ X\right)  }{du}\!=\!%
			 \okap_{1}\left(  y_{2}^{i}\circ X\right)
			\!-\!\Gamma_{jk}^{i}\dfrac{d\left(  x^{j}\circ\alpha\right)  }{du}\left(
			y_{1}^{k}\circ X\right)   ,\smallskip\\
			\dfrac{d\left(  y_{2}^{i}\circ X\right)  }{du}\!=\!  -\okap_{1}\left(  y_{1}^{i}\circ X\right)  \!+\!\okap_{2}\left(  y_{3}^{i}\circ X\right)
			\!-\!\Gamma_{jk}^{i}\dfrac{d\left(  x^{j}\circ\alpha\right)  }{du}\left(
			y_{2}^{i}\circ X\right)    ,\smallskip\\
			\dfrac{d\left(  y_{3}^{i}\circ X\right)  }{du}\!=\!  -\okap_{2}\left(  y_{2}^{i}\circ X\right)  \!-\!\Gamma_{jk}^{i}\dfrac{d\left(  x^{j}\circ\alpha\right)  }%
			{du}\left(  y_{3}^{i}\circ X\right)    ,\smallskip
		\end{array}
		\right.  \label{sistemaEDO}%
	\end{equation}
	with $1\leq i\leq3$ and where $\Gamma_{jk}^{i}$ are the components of the
	Levi-Civita connection $\nabla$ of $g$ with respect to the coordinate system
	$(x_{i})_{i=1}^{3}$.
	From the general theory of
	ODE's, the system (\ref{sistemaEDO}) has unique solutions $x^{i}\circ\alpha
	,y_{k}^{j}\circ X$, $i,j,k=1,2,3$, satisfying the initial conditions $\left(  x^{i}\circ\alpha\right)
	(u_{0})=x^{i}(x_{0})$, $\left(  y_{k}^{j}\circ X\right)  (u_{0}%
	)=\hat{\delta}_{k}^{j}$, $i,j,k=1,2,3$, where $\hat{\delta}$ is the Kronecker
	delta. We consider the ruled surface $\mathbf{X}\colon I\times\mathbb{R}\rightarrow (M,g)$, $\mathbf{X}(u,v)=\exp_{\alpha(u)}(vZ(u))$
	with
	\[
	\alpha(u)=(x^{1}(u),x^{2}(u),x^{3}(u)),\qquad Z(u)=y_{1}^{i}(u)\left(
	\frac{\partial}{\partial x^{i}}\right)  _{\alpha(u)}.
	\]
	We first prove that
	\[
	X_{k}(u)=y_{k}^{i}(u)\left(  \frac{\partial}{\partial x^{i}}\right)
	_{\alpha(u)},\qquad1\leq k\leq3,
	\]
	define an orthonormal basis. To this end, consider the functions
	\[
	\phi_{ij}(u)=g_{\alpha(u)}(X_{i}(u),X_{j}(u)),\qquad1\leq i\leq j\leq3.
	\]
	As a consequence of the system (\ref{sistemaEDO}), it is straightforward to check that $\phi_{ij}$ satisfy the
	following equations,
	\[
	\left\{
	\begin{array}
		[c]{l}%
		\dfrac{d\phi_{11}}{du}=2\okap_{1}\phi_{12},\smallskip\\
		\dfrac{d\phi_{12}}{du}=\okap_{1}\phi_{22}-\okap_{1}%
		\phi_{11}+\okap_{2}\phi_{13},\smallskip\\
		\dfrac{d\phi_{13}}{du}=\okap_{1}\phi_{23}-\okap_{2}%
		\phi_{12},\smallskip\\
		\dfrac{d\phi_{22}}{du}=-2\okap_{1}\phi_{12}+2\okap_{2}%
		\phi_{23},\smallskip\\
		\dfrac{d\phi_{23}}{du}=-\okap_{1}\phi_{13}+\okap_{2}%
		\phi_{33}-\okap_{2}\phi_{22},\smallskip\\
		\dfrac{d\phi_{33}}{du}=-2\okap_{2}\phi_{23},
	\end{array}
	\right.
	\]
	with initial conditions $\phi_{ij}(u_{0})=\hat{\delta}_{j}^{i}$. But the
	constant function $\hat{\phi}_{ij}(u)=\hat{\delta}_{j}^{i}$ also satisfies this system of differential equations with these initial conditions. By the uniqueness theorem of EDOs, it follows that $\phi_{ij}(u)=\hat{\delta}_{j}^{i}$, for
	all $u$. In addition, since $(  X_{1}(u_{0}),X_{2}(u_{0}),X_{3}	(u_{0}))  $ is an oriented positive basis, by continuity, so it is $(X_{1}(u),X_{2}(u),X_{3}(u))$ for any $u$. The frame determined by $X$ verifies
	that $X_{1}=Z$ and $\nabla_{\alpha^{\prime}}X_{1}=\okap_1 X_{2}$, so since $(
	X_{1},X_{2},X_{3})$ is positive, it is the Sannia basis of the ruled
	surface $\mathbf{X}\colon I\times\mathbb{R}\rightarrow (M,g)$. Finally, again from (\ref{sistemaEDO}), it follows that
	$\okap_{1}$ and $\okap_{2}$ are the Sannia invariants.
	
	We now prove the uniqueness of the ruled surface up to isometries. First, if
	$(M,g)$ is of constant curvature, given $p_{0},p_{0}^{\prime}\in M$ and
	$( e_{i})_{i=1}^3$, $(e'_{i})_{i=1}^3$ oriented orthonormal bases at $T_{p_{0}}M$ and
	$T_{p_{0}^{\prime}}M$ respectively, we consider a local isometry $\phi$
	sending $p_{0}$ to $p_{0}^{\prime}$ and $( e_{i} )_{i=1}^3$
	to $(e'_{i})_{i=1}^3$. Since
	$\phi$ preserves the Levi-Civita connection, the image $\phi\circ\mathbf{X}$
	of the ruled surface defined by $p_{0}$ and $( e_{i} )_{i=1}^3$
	is a ruled surface with the same parameters $\okap_{1}$,
	$\okap_{2}$. By the uniqueness of the system of differential
	equations \eqref{sistemaEDO}, we have: $\mathbf{X}^{\prime}=\phi\circ\mathbf{X}$.
	Conversely, if given $p_{0}$, $p_{0}^{\prime}\in M$ and orthonormal bases
	$( e_{i} )_{i=1}^3$, $(e'_{i})_{i=1}^3$
	at $T_{p_{0}}M$ and $T_{p_{0}^{\prime}}M$ respectively, there is a local isometry sending one to the other,
	the space is locally isotropic, and hence of constant curvature (\emph{cf.} \cite{KN}).
\end{proof}

\section{Central points and striction curves}
\label{Strictioncurves}

Euclidean ruled surfaces contain a special curve called \emph{striction curve} which is the locus of the points whose distance with respect to neighbouring geodesics is extremal. In particular striction curves are divided into \emph{expanding-contracting} and \emph{contracting-expanding}
curves depending on whether this distance corresponds to a maximum or a minimum value respectively. In general, the striction curve of a family of Euclidean curves is the locus for the points where the geodesic curvature of the corresponding orthogonal set of curves vanishes (read \cite{Muller1941,Muller1948} for more details). The striction curve of any Euclidean ruled surface is also characterised by the fact that the generators along it are parallel in the Levi-Civita sense. The following definition extends such a property to general ambient Riemannian manifolds:
\begin{definition}
	Let $\mathbf{X}\colon I\times\mathbb{R}\rightarrow (M,g)$ be a parametrized ruled surface. A curve $s\colon I \to \mathbf{X}(I\times \mathbb{R})$
    is said to be a \emph{striction curve} if
\begin{equation*}
	g(s^{\prime},\nabla_{s^{\prime}}\xv)=0. \label{strictionrelation}
\end{equation*}
\end{definition}

As already mentioned, whenever the striction curve $s\colon I \to \mathbf{X}(I\times \mathbb{R})$ exists, it can be taken as the base curve of the ruled surface so that it can be reparametrized as $\mathbf{X}(u,v)=\mathrm{exp}_{s(u)}(vZ(u))$.

\begin{proposition}
	\label{Sanniatangentvec}
	Let  $\mathbf{X}\colon I\times\mathbb{R}\rightarrow (M,g)$ be a Sannia ruled surface in a $3$-dimensional Riemannian manifold $\left(M,g\right)$ admitting a striction curve $s\colon I \to \mathbf{X}(I\times \mathbb{R})$ and chosen to be its base curve. For any $p=\mathbf{X}(u,0)$, the vectors $\{X_{1},X_{3}\}$ of the associated Sannia frame constitute an orthonormal basis of $T_{p}\,(\mathbf{X}(I\times\mathbb{R}))$ and the angle $\varphi$ defined by (\ref{basecurvedecomposition}) vanishes identically. Therefore
\begin{equation*}
	s^{\prime}(u)=\cos\sigma_{s(u)}\,X_{1}+\sin\sigma_{s(u)}\,X_{3} \quad \text{for all $u\in I$}.
\end{equation*}

\end{proposition}

\begin{proof}
	The above relation holds since the tangent plane to  $\mathbf{X}\colon I\times\mathbb{R}\rightarrow (M,g)$ along $s$ is spanned by $s^{\prime}$ and $Z$, and
	the second vector $X_2=\nabla_{s'}Z/\vert\vert \nabla_{s'}Z \vert\vert$ of the corresponding Sannia frame is perpendicular to both of them.
\end{proof}

As mentioned above, in the Euclidean setting the ruling vector field of a ruled surface is parallel along the striction curve with respect to its induced Levi-Civita connection. This property also holds in a general Riemannian background as shown in the following result.
\begin{corollary}
	Under the hypotheses of Proposition \ref{Sanniatangentvec}, the ruling vector field $Z$ is parallel along the striction curve with respect to the induced connection on the Sannia ruled surface.
\end{corollary}	
\begin{proof}
	The Gauss identity on the striction curve reads
	\begin{equation*}
 		\nabla_{s'}Z =\nabla^{\Sigma}_{s'}Z+\vec{h}(s',Z),
    \end{equation*}
where $\nabla^{\Sigma}$ stands for the induced connection on the ruled surface $\Sigma\equiv \mathbf{X}(I\times\mathbb{R})$. By Proposition \ref{Sanniatangentvec}, $\nabla_{s'}Z$ is orthogonal to $\Sigma$, which means that $\nabla^{\Sigma}_{s'}Z=0$.
\end{proof}

\begin{remark}
    In the classical statement of the fundamental theorem or ruled surfaces
    in the Euclidean space $(\mathbb{R}^3,\delta)$, a striction curve can always be taken as the base curve in the parametrization for it always exists. By virtue of the above proposition, in such a context $\varphi=0$ and $\theta=\sigma$, i.e. the base angle on the striction curve.
	In particular, the distribution parameter formula (\ref{basecurvdistribparam}) on a general base curve reduces to
    \begin{equation*}
  	\lambda(u,0)=\frac{\sin{\sigma_{s(u)}}}{\kappa_1(u)},
    \end{equation*}
	which is the classical expression for the Euclidean distribution parameter (on a striction curve)
	\cite[Lemma 5.3.7]{PottmannWallner}.

\end{remark}

In the following result we obtain a condition for the striction curve in terms of the  coordinate basis $(\xu,\xv)$.

\begin{proposition}
	\label{criticalJac}
	Let $\mathbf{X}\colon I\times\mathbb{R}\rightarrow (M,g)$ be a
	parametrized ruled surface $\mathbf{X}\colon I\times\mathbb{R}\rightarrow (M,g)$ with $\alpha\colon I \to \mathbf{X}(I\times \mathbb{R})$ an associated base curve, and with a unit ruling vector field $Z\in \mathfrak{X}(\alpha)$. A curve $s\colon I \to \mathbf{X}(I\times \mathbb{R})$ on such a surface is a
	striction curve if and only if the squared norm of the vector field $\xu$ along each generator $\gamma_{Z(u)}(v)$
	is critical at the corresponding point $s(u)$, i.e.
	\begin{equation}
		\label{criticalJacobi}
		\xv(\Vert \xu\Vert ^{2})|_{p}=0,\quad\text{for all $p=s(u)$}.
	\end{equation}
	\end{proposition}

\begin{proof}
Given any curve of the form $s\colon I\rightarrow \mathbf{X}(I\times \mathbb{R})$, $s(u)=\mathbf{X}(u,v(u))$, where $v\colon I \rightarrow \mathbb{R}$ is smooth, we have $s^{\prime}=\xu+v^{\prime}\xv \label{tanstrict}$. Then
\begin{eqnarray*}
g(s^{\prime} ,\nabla _{s^{\prime}}\xv)&=& g (\xu+v^{\prime}\xv, \nabla _{\xu+v^{\prime}\xv} \xv)= g (\xu+v^{\prime}\xv, \nabla _{\xu} \xv )\\
&=& g (\xu, \nabla _{\xu} \xv ) = g (\xu, \nabla _{\xv} \xu )
= \tfrac{1}{2}\xv(\Vert \xu\Vert ^{2}),
\end{eqnarray*}
where we have taken into account that $g(\xv,\xv)$ is constantly $1$, and that $\nabla _{\xu}\xv-\nabla _{\xv}\xu=[\xu,\xv]=0$. The proof is complete by the definition of striction curve.
\end{proof}

\begin{remark}
	Since the length of the Jacobi vector field $\xu$ gives a local idea of
	the deviation of a congruence of geodesics,
	relation (\ref{criticalJacobi}) is in accordance with the
	definition of striction line established in a Euclidean sense, where
	any point in these curves has a critical distance with respect to neighbouring geodesics, as mentioned before.
\end{remark}

The previous result motivates the following definition.
\begin{definition}
	Let $\mathbf{X}\colon I\times\mathbb{R}\rightarrow (M,g)$ be a parametrized ruled surface in a $3$-dimensional Riemannian manifold $(M,g)$. The function
	that describes the derivative of the squared norm of the Jacobi field
	\begin{equation}
		\label{Jacobifun}
		F\colon \mathbf{X}(I\times \mathbb{R}) \rightarrow \mathbb{R}, \quad F(p)=\tfrac{1}{2}(\xv \Vert\xu\Vert^2)|_p
	\end{equation}
	will be referred to as the \emph{Jacobi evolution function} of the ruled surface. The points of the parametrized ruled surface where $F$ vanishes are called \emph{central points}.
\end{definition}
\begin{remark}
	\label{Fstrictioncondition}
	Therefore, striction curves lie on the vanishing set of the function $F$.
Furthermore, the function $F$ describes the evolution of the norm of the Jacobi variational field $\xu$ along each ruling.
The study of the sign of $F$ provides relevant information with regard to the behavior of the rulings in a local manner.
If $F$ is strictly positive at a point, the norm of the Jacobi is increasing and hence the rulings are locally spreading apart.
Just the opposite happens in case that $F$ is strictly negative at some point. Since $(u,v)$ is a system of coordinates associated with the parametrization
(\ref{rulparam}) of the ruled surface, we will
simply refer to $(F\circ \mathbf{X})(u,v)$ as $F(u,v)$ for the sake of simplicity.
\end{remark}

We next explore the behaviour of the first and second derivatives of the function $F$.

\begin{theorem}
	\label{ThFderivatives}	
	Let $\mathbf{X}\colon I\times\mathbb{R}\rightarrow (M,g)$ be a parametrized ruled surface in a Riemannian $3$-manifold $(M,g)$. The first and second derivatives of the
	Jacobi evolution function $F$ are
	\begin{eqnarray}
		&& \frac{\partial F}{\partial v}=-\Riem(\xv,\xu,\xv,\xu)+\Vert\nabla_{\xv} \xu\Vert^2, \label{Jacfirstder} \\
				&&\frac{\partial^2 F}{\partial \, v^2}=
		g\left(\left(\con_{\xv}R\right)(\xv,\xu,\xv),\xu\right)	
		-4\Riem(\xv,\xu,\xv,\conv \xu). \label{Jacsecondder}
		\end{eqnarray}
	\end{theorem}
\begin{proof}
	If we differentiate $F$ along the $\xv$-direction, we obtain
	\begin{equation*}
		\frac{\partial F}{\partial v}=\conv g(\conv \xu,\xu)=
		g(\conv\conv \xu,\xu)+g(\conv \xu,\conv \xu).
	\end{equation*}
	As $\xu$ is a Jacobi vector field along the rulings, it satisfies (\ref{Jacobeq}) so that
	\begin{eqnarray*}
		\frac{\partial F}{\partial v}&=&-g(R(\xu,\xv)\xv,\xu)+\Vert\conv \xu\Vert^2  \nonumber \\
		&=&-\Riem(\xv,\xu,\xv,\xu)+\Vert\nabla_{\xv} \xu\Vert^2, \label{firtderprov}
    \end{eqnarray*}
	and
	\begin{eqnarray}
		\frac{\partial^2 F}{\partial \, v^2}&=&g\left(\conv(R(\xv,\xu)\xv),\xu\right)+g(R(\xv,\xu)\xv),\conv \xu) \label{secondderprov} \\
		&&+2g(\conv \conv \xu,\conv \xu)  \nonumber \\
				&=&g\left(\conv(R(\xv,\xu)\xv),\xu\right)+3g(R(\xv,\xu)\xv,\conv \xu).	 \nonumber		
	\end{eqnarray}
	On the other hand, taking into account the symmetries of the Riemann tensor and the fact that $\xv$ is a geodesic vector field we obtain
	\begin{eqnarray}	
		g(\conv(R(\xv,\xu)\xv),\xu)&=&g((\conv R)(\xv,\xu,\xv),\xu ) \nonumber \\
		                           & &+g(R(\xv,\xu)\xv,\conv \xu). \label{derRiemann}
	\end{eqnarray}	
	Plugging (\ref{derRiemann}) into (\ref{secondderprov}) finally gives (\ref{Jacsecondder}), which finally concludes the proof of the theorem.
\end{proof}

As a consequence, in the special case where the ambient manifold has constant curvature the above derivatives of $F$ read as follows.
\begin{corollary}
	\label{corJacderspaceform}
	Let $(M(k),g)$ be a Riemannian $3$-manifold of constant sectional curvature $k$ and  $\mathbf{X}\colon I\times\mathbb{R}\rightarrow (M,g)$ be a parametrized ruled surface with $Z$ a unit vector field. Then the first and second derivatives of the Jacobi evolution function $F$ read
	\begin{eqnarray}
		\label{Jacfirstspaceform}
		\frac{\partial F}{\partial v}&=&-k(\Vert\xu\Vert^2-\Vert\alpha^\prime(u)\Vert^2\cos^2\sigma_p)+\Vert\conv \xu\Vert^2,
\end{eqnarray}
and
\begin{eqnarray}
		\frac{\partial^2 F}{\partial \, v^2}&=& (-2k)\xv(\Vert\xu\Vert^2)=-4kF(u,v),
\label{Jacsecondspaceform}
	\end{eqnarray}
	where $p=\alpha(u)=\mathbf{X}(u,0)$ and $\sigma_p$ is the base angle between $\alpha^\prime(u)$ and $Z$ at $p$.
\end{corollary}
\begin{proof}
	By the sectional curvature relation for a space form, we have
	\begin{equation}
		\label{seccurvature}
		\Riem(\xv,\xu,\xv,\xu)=k\left(\Vert\xu\Vert^2-g(\xu,\xv)^2 \right),
	\end{equation}	
	since $\xv$ is a unit vector field. On the other hand, we know by Proposition \ref{Propcoordframeangle} that $g(\xu,\xv)=\Vert\alpha^\prime(u)\Vert\cos{\sigma_p}$ is constant along the associated ruling. Hence, (\ref{seccurvature}) can be rewritten as
	\begin{equation}
		\label{seccurvaturetwo}
		\Riem(\xv,\xu,\xv,\xu)=k\left(\Vert\xu\Vert^2-\Vert\alpha^\prime(u)\Vert^2 \cos^2 \sigma_p \right),
	\end{equation}
	so (\ref{Jacfirstder}) becomes (\ref{Jacfirstspaceform}) when relation (\ref{seccurvaturetwo}) is inserted.
	Let us derive now relation (\ref{Jacsecondspaceform}). By virtue of formula (\ref{Jacsecondder}), the second derivative of the Jacobi evolution function $F$ in a space form becomes
	\begin{equation*}
		\frac{\partial^2 F}{\partial \, v^2}=-4\,\Riem(\xv,\xu,\xv,\conv \xu)=4\,g(R(\xv,\xu)\xv,\nabla_{\xv}\xu),
	\end{equation*}
	for $\nabla R=0$ holds. Since
	\begin{equation*}
		R(\xv,\xu)\xv=R(\xv,\xu^{\bot})\xv=-k \xu^{\bot},
	\end{equation*}
	where the last equality is fulfilled as a consequence of the ambient manifold being a space form, we have
		\begin{equation}
		\label{nearlysecondderiv}
		\frac{\partial^2 F(u,v)}{\partial v^2}=(-4k) g(\xu^{\bot},\nabla_{\xv}\xu).
	\end{equation}
	Using the splitting given in Proposition \ref{Propjacdecomposition}  for the Jacobi field $\xu$, we obtain
	\begin{eqnarray}
		\label{splittingterm}
		g(\xu^{\bot},\nabla_{\xv}\xu)&=&g(\xu,\nabla_{\xv}\xu)-\Vert\alpha^\prime(u)\Vert(\cos{\sigma_p})g(\xv,\nabla_{\xv}\xu)\\ &=&\frac{1}{2}\xv(\Vert\xu\Vert^2), \nonumber
	\end{eqnarray}
    with $\sigma_p$ the base angle at $p=\alpha(u)$, where $[\xu,\xv]=0$ has been applied in combination with the fact that
	$\xv$ is geodesic. Inserting (\ref{splittingterm}) into (\ref{nearlysecondderiv}) finally gives (\ref{Jacsecondspaceform}).
\end{proof}

In the next result we integrate the differential equation  (\ref{Jacsecondspaceform}) in the different models for positive, zero or negative $k$.
\begin{theorem}
	 Let  $\mathbf{X}\colon I\times\mathbb{R}\rightarrow (M,g)$ be a parametrized ruled surface in a complete manifold $(M,g)$ of constant sectional curvature $k$, with $Z$ a unit ruling vector field.  Then the expression of the Jacobi evolution function $F$ is
	\begin{eqnarray}
		\label{Jacobifuncsphere}
		F(u,v)&=& C_1 \cos{(\sqrt{4k}v)}+C_2 \sin{(\sqrt{4k}v)} , \quad \text{if $k>0$}, \\
		\label{Jacobifunceuclid}
		F(u,v)&=&C_1 + C_2 v, \quad \text{if $k=0$} \\
		\label{Jacobifunchyperbolic}
		F(u,v)&=&C_1 \cosh{(\sqrt{-4k}v)}+C_2\sinh{(\sqrt{-4k}v)}, \quad \text{if $k<0$}
	\end{eqnarray}
	where
\begin{equation}
			\label{C1}
			C_1=g_p(\nabla_{\alpha^\prime}Z,\alpha^\prime),
\end{equation}
and
\begin{eqnarray}
		\label{Cspherehyperbolic}
		C_2&=&\tfrac{1}{\sqrt{|4k|}}\left( -k\Vert\alpha^\prime\Vert_p^2\sin^2{\sigma_p}+\Vert\nabla_{\alpha^\prime}Z\Vert_p^2\right) , \quad \text{if $k>0$ or $k<0$}, \\
		\label{Ceuclid}
		 C_2&=&\Vert\nabla_{\alpha^\prime}Z\Vert^2_p,   \quad  \text{if $k=0$}.
\end{eqnarray}

\end{theorem}
\begin{proof}
	Evaluating the Jacobi evolution function $F(u,v)$ at $v=0$ gives
	\begin{equation*}
		F(u,0)=\tfrac{1}{2}\xv(\Vert\xu\Vert^2)|_p=g_p(\nabla_{\xv}\xu,\xu)=g_p(\nabla_{\xu}\xv,\xu)=g_p(\nabla_{\alpha^\prime}Z,\alpha^\prime),
	\end{equation*}
	which implies  (\ref{C1}), after evaluating  (\ref{Jacobifuncsphere}),  (\ref{Jacobifunceuclid}) and  (\ref{Jacobifunchyperbolic}), at $v=0$. To compute $C_2$, we use the Jacobi first derivative relation (\ref{Jacfirstspaceform}), which becomes
	\begin{equation}
		\label{Fderivativebasecurve}
		\frac{\partial F}{\partial v}(u,0)=-k\Vert\alpha^\prime(u)\Vert_p^2\sin^2{\sigma_p}+\Vert\nabla_{\alpha^\prime} Z\Vert^2_p
	\end{equation}
	since $\xu(u,0)=\alpha^\prime(u)$. Differentiating (\ref{Jacobifuncsphere}) at $v=0$, from (\ref{Fderivativebasecurve}), we obtain
	\begin{equation*}
		C_2 \sqrt{4k}=-k\Vert\alpha^\prime(u)\Vert^2\sin^2{\sigma_p}+\Vert\nabla_{\alpha^\prime} Z\Vert^2_p,
	\end{equation*}
	which is none other than $C_2$ in (\ref{Cspherehyperbolic}) for $k>0$. The respective value of $C_2$ for $k<0$ is obtained in an analogous way differentiating  (\ref{Jacobifunchyperbolic})  at $v=0$.
	Likewise, in the Euclidean context $k=0$, $(\partial F/\partial v)(0)=\Vert\nabla_{\alpha^\prime} Z\Vert^2_p$, which implies the value of $C_2$ in (\ref{Ceuclid}) for the Jacobi evolution function is $F(u,v)=C_1+C_2 v$  in the flat case.
\end{proof}

\begin{theorem}
	Let $\mathbf{X}\colon I\times\mathbb{R}\rightarrow (M,g)$ be a
	Sannia ruled surface $\mathbf{X}\colon I\times\mathbb{R}\rightarrow (M,g)$ in a complete manifold $(M,g)$ of constant sectional curvature $k$, where $\alpha\colon I \to \mathbf{X}(I\times \mathbb{R})$ is an associated base curve, and with a unit ruling vector field $Z\in \mathfrak{X}(\alpha)$. If $s\colon I \to \mathbf{X}(I\times \mathbb{R})$, $s(u)=\mathbf{X}(u,v(u))$ is an striction curve, then
	\begin{eqnarray}
		\label{strictionsphere}
		v(u)&=&\frac{1}{\sqrt{4k}}\arctan{\left(\frac{-\sqrt{4k}\,g_p(\nabla_{\alpha^\prime}Z,\alpha^\prime)}{-k\Vert\alpha^\prime\Vert_p^2\sin^2{\sigma_p}+\Vert\nabla_{\alpha^\prime} Z\Vert^2_p}\right)} \quad \text{if $k>0$}, \\
		\label{strictioneuclid}
		v(u)&=& -\frac{g_p(\nabla_{\alpha^\prime}Z,\alpha^\prime)}{\Vert\nabla_{\alpha^\prime}Z\Vert^2_p}  \quad \text{if $k=0$},  \\
		\label{strictionhyperbolic}
		 v(u)&=&\frac{1}{\sqrt{-4k}}\mathrm{arctanh}{\left(\frac{-\sqrt{-4k}\,g_p(\nabla_{\alpha^\prime}Z,\alpha^\prime)}{-k\Vert\alpha^\prime\Vert_p^2\sin^2{\sigma_p}+\Vert\nabla_{\alpha^\prime} Z\Vert^2_p}\right)} \quad \text{if $k<0$},
	\end{eqnarray}
	where $\sigma_p$ is the angle between the vectors $\alpha^\prime$ and $Z$ at $p=\alpha(u)$. In addition, if $M$ is simply connected (that is, $M=S^3$ for $k>0$, $M=\mathbb{R}^3$ for $k=0$ or $M=\mathbb{H}^3$ for $k<0$ with their standard metrics scaled by $1/\sqrt{|k|}$ when  $k\neq 0$), then for $k\geq 0$, every Sannia surface has a unique striction curve. For $k<0$, every Sannia surface has at most one striction curve.
\end{theorem}

\begin{proof}
	As already noted in Remark \ref{Fstrictioncondition}, a necessary condition for a point $p$ to lie on the striction curve is that the Jacobi evolution function verifies $F(p)=0$. Then the expressions for $v(u)$ are directly obtained from \eqref{Jacobifuncsphere}, \eqref{Jacobifunceuclid} and \eqref{Jacobifunchyperbolic} respectively, taking into account \eqref{Cspherehyperbolic} or \eqref{Ceuclid}.
    For $k>0$, the expression must be understood to provide $v(u)=\pm\pi /2$ if $-k\Vert\alpha^\prime\Vert_p^2\sin^2{\sigma_p}+\Vert\nabla_{\alpha^\prime} Z\Vert^2_p=0$. In any case, the value of $v(u)$ is periodic with period $2\pi/\sqrt{k}$, which is exactly the period of the geodesics in $S^3$, so that the striction curve is geometrically unique. The uniqueness holds trivially true for $k=0$. For $k<0$, the injectivity of the function $\mathrm{arctanh}$ implies that there is at most a solution for $v(u)$.
\end{proof}

\begin{remark}
	Formula \eqref{strictioneuclid} is the classical formula for the striction curve in the Eucliden space. Formulas (\ref{strictionsphere}) in $S^3$ and (\ref{strictionhyperbolic}) in $\mathbb{H}^3$ were given in \cite{DaSilvaDaSilva2022}, where the uniqueness of striction curves is also addressed. However, they obtain the results by working with the sphere or the hyperbolic space as submanifolds of $\mathbb{R}^4$ with the standard Euclidean or Lorentzian metric respectively, that is, in a less intrinsic fashion. Our approach involves a differential equation that, in principle, could be analysed in arbitrary Riemannian manifolds. Note that in such a case, the equation will involve the curvature of the ambient space. Finally, the authors also claim in \cite{DaSilvaDaSilva2022} that the striction curve always exists for $k<0$, a fact that is wrong as the following example illustrates.
\end{remark}

 \textit{Example 1: A ruled surface in the hyperbolic space without striction curve.} We consider in $(\mathbb{H}^3,g=\frac{1}{z}(dx^2+dy^2+dz^2))$ the ruled surface with base curve $\alpha\colon [0,2\pi)\rightarrow (\mathbb{H}^3,g)$ determined by $\alpha(u)=(\cos{u},\sin{u},1)$ and the unit ruling vector $Z(u)=(\cos{u},\sin{u},0)$ defined along $\alpha$. A straightforward computation shows that
	\begin{equation*}
		\nabla_{\alpha^\prime}Z=-\sin{u}\,\partial_x+\cos{u}\,\partial_y.
	\end{equation*}
	The form for the Jacobi evolution function $F$ is described by relation (\ref{Jacobifunchyperbolic})
	for negative constant sectional curvature. The coefficients $C_1$ and $C_2$ read this time $C_1=g_p(\alpha^\prime,\nabla_{\alpha^\prime}Z)=1$ and
	$C_2=\frac{1}{2}\left( \Vert\alpha^\prime\Vert_p^2\sin^2{\sigma_p}+\Vert\nabla_{\alpha^\prime}Z\Vert_p^2\right)=1$, since $\sigma_p=\pi/2$.
	Therefore, the Jacobi evolution function associated with such a ruled surface is
	\begin{equation*}
		F(u,v)=\cosh{2v}+\sinh{2v}.
	\end{equation*}
	Condition $F=0$ holds on the striction line in case of existence, which is equivalent to
	\begin{equation*}
		1+\tanh{2v}=0,
	\end{equation*}
	which clearly does not have any solution. \\

\vspace{3mm}
Actually, we can characterize when a (complete) Sannia ruled surface in $\mathbb{H}^3(k)$ never admits a striction curve in terms of its extrinsic curvature and first Sannia invariant. For that, first note that in a ruled surfaces defined by an arc-parametrized curve $\alpha$ and a unitary vector field $Z(u)=\pm \alpha '(u)$ (that is, tangent ruled surfaces), the curve $\alpha$ is already a striction curve. On the other hand, if $\alpha '(u)$ and $Z(u)$ are linearly independent, the curve $\bar{\alpha}(u)=\mathbf{X}(u,f(u))$ where $f$ satisfies the ODE $f'(u)=-g(\mathbf{X}_u,\mathbf{X}_v)_{\mathbf{X}(u,f(u))}$, defines a new parametrization of the same surface such that $\bar{\alpha}'(u)$ and $\bar{Z}(u)=\mathbf{X}_v(u,f(u))$ are always orthogonal.

 \begin{proposition}\label{ship}
	Let $\mathbf{X}(u,v)=\exp _{\alpha (u)}(vZ(u))$ a Sannia ruled surface in $\mathbb{H}^{3}(k)$
	defined by an arc-parametrized curve $\alpha $ and a unitary ruling vector
	field $Z\in \mathfrak{X}(\alpha )$ such that $g(\alpha ^{\prime },Z)=0$.
	Then the surface does not admit a striction curve if and only if its
	extrinsic curvature $K_{ext}$ vanishes and the first Sannia curvature satisfies $\kappa _{1}=\sqrt{-k}$.
\end{proposition}

\begin{proof}
	According to \eqref{strictionhyperbolic}, the inexistence of the striction curve is equivalent to the condition%
	\begin{equation*}
		\frac{2\sqrt{-k}\left\vert g(\alpha ^{\prime },\nabla _{\alpha ^{\prime
			}}Z)\right\vert }{-k+\left\Vert \nabla _{\alpha ^{\prime }}Z\right\Vert ^{2}}%
		=\frac{2\sqrt{-k}\left\Vert \nabla _{\alpha ^{\prime }}Z\right\Vert
			\left\vert \cos \psi \right\vert }{-k+\left\Vert \nabla _{\alpha ^{\prime
			}}Z\right\Vert ^{2}}\geq 1,  \label{2}
	\end{equation*}
	where $\psi =\widehat{\alpha ^{\prime },\nabla _{\alpha ^{\prime }}Z}$. On
	the other hand, we always have
	\[
	\frac{2\sqrt{-k}\left\Vert \nabla _{\alpha ^{\prime }}Z\right\Vert
		\left\vert \cos \psi \right\vert }{-k+\left\Vert \nabla _{\alpha ^{\prime
		}}Z\right\Vert ^{2}}\leq \frac{2\sqrt{-k}\left\Vert \nabla _{\alpha ^{\prime
		}}Z\right\Vert }{-k+\left\Vert \nabla _{\alpha ^{\prime }}Z\right\Vert ^{2}}%
	\leq 1 ,
	\]%
	where last step is the arithmetic-geometric mean inequality. Then, we have
	that expression (2) is never bigger than 1 and it is 1 if and only if $\cos
	\psi =\pm 1$ and $\left\Vert \nabla _{\alpha ^{\prime }}Z\right\Vert =\sqrt{-k}$.
	The proof is now complete by taking into account that, on one hand $Z=X_{1}$
	and $\kappa _{1}=\left\Vert \nabla _{\alpha ^{\prime }}X_{1}\right\Vert $,
	and on the other hand noting that $vol(\alpha ^{\prime },Z,\nabla _{\alpha
		^{\prime }}Z)$ vanishes (which is equivalent to the vanishing of $K_{ext}$)
	if and only if $\psi =0$ or $\pi $ (that is, $\alpha ^{^{\prime }}$ is
	parallel to $\nabla _{\alpha ^{\prime }}Z$), since $\alpha ^{\prime }\perp Z$
	, and $Z\perp \nabla _{\alpha ^{\prime }}Z$.
\end{proof}

We end the article in ambient manifolds of non-constant curvature, where we see that the uniqueness of the striction curve cannot be guaranteed.

 \vspace{4mm}

  \textit{Example 2: A ruled surface in a product manifold equipped with different striction curves.} We consider the surface of revolution $\psi\colon\Sigma\rightarrow (\mathbb{R}^3,\delta)$ in $(\mathbb{R}^3,\delta)$ generated by the rotation of the curve $f(x)=2+\sin{x}$ in the $xy$-plane about
	the $x$-axis. Such a surface is defined by the following parametrization:
	\begin{eqnarray}
				&&\mathbf{X}\colon [0,2\pi)\times \mathbb{R}\rightarrow (\mathbb{R}^3,\delta) \nonumber \\
&& \mathbf{X}(u,v)=(v, (2+\sin{v})\cos{u}, (2+\sin{v})\sin{u}) . \label{paramseveralstriction}
	\end{eqnarray}
	The expression for the induced metric $g_\Sigma$ on $\Sigma$ reads
	\begin{equation*}
		g_\Sigma=(2+\sin{v})^2 du^2+(1+\cos^2{v})dv^2
	\end{equation*}
	in local coordinates associated with the parametrization (\ref{paramseveralstriction}). It is a well-known fact that
	the generating curve of a surface of revolution in $(\mathbb{R}^3,\delta)$ is a geodesic of the surface.
	Consider the product manifold $\mathbb{R}\times \Sigma$ endowed with the product metric $g=dt^2+g_\Sigma$. Since
	$\Sigma\subset (\mathbb{R}\times\Sigma,g)$ is a totally geodesic slice of this Riemannian product, its Gauss
	curvature is zero and the generating geodesic curves of $(\Sigma,g_\Sigma)$ are also geodesics in the ambient
	product manifold. Let us choose $\alpha(u)=\mathbf{X}(u,0)$ as base curve
	of $(\Sigma,g_\Sigma)$.
	When we consider the embedded ruled surface $\psi\colon \Sigma\rightarrow (\mathbb{R}\times\Sigma,g)$,
	$\alpha^\prime(u)=\xu |_\alpha$ is the initial $v=0$ value  of the Jacobi vector field $\xu$ along the rulings, and $Z=(1/\sqrt{1+\cos^2{v}})\xv$ is a unit geodesic vector tangent to them. Let us compute the Jacobi evolution function $F$  in this context. A straightforward calculation gives
	\begin{equation*}
		\nabla_{\alpha^\prime}Z=\frac{\cos{v}}{(1+\cos^2{v})(2+\sin{v})}\xu.
	\end{equation*}
	As a consequence
	\begin{equation*}
		F(u,v)=g(\xu,\nabla_{\xu}Z)=\frac{\cos{v}(2+\sin{v})}{1+\cos^2{v}}.
	\end{equation*}
	The striction curves are given by the solution of $F=0$, and this happens if and only if $v=\pi/2+k\pi$, where $k\in \mathbb{Z}$. This means that each curve $\alpha_{v_k}(u)=(u,\pi/2+k\pi)$ with $k\in \mathbb{Z}$ is a striction curve
	of $\psi\colon \Sigma\rightarrow (\mathbb{R}\times\Sigma,g)$.  \\

  \textit{Example 3: A ruled surface in a warped product manifold with an arbitrary number of striction curves.} Given an open interval  $I\subset \mathbb{R}$ and a $2$-dimensional Riemannian manifold $(F,g_F)$, consider the product $3$-manifold $I\times F$ endowed with the metric $g=dt^2+f^2(t)g_F$, where $f\colon I\rightarrow \mathbb{R}$ is a smooth positive function. We will refer to the warped product manifold $(I\times F,g)$ as $I\times_f F$. Consider a closed unit curve $\alpha^F\colon[a,b)\rightarrow (F,g_F)$ in the fiber of $I\times_f F$.
	Given any $v\in I$, $\alpha^F$ can be lifted in a natural way to the slice $\{t=v\}$ as the curve
	$\alpha_v\colon [a,b)\rightarrow I\times_f F$ defined by $\alpha_v(u)=(v,\alpha^F(u))$, with  $Z=\partial_t$ as ruling unit vector field along $\alpha_v$. As usual, the initial value of the Jacobi field $\xu$ on $\alpha_v$ is $\xu|_{\alpha_v}=\alpha^\prime_v$. Since $Z=\partial_t|_{\alpha_v}$ is orthogonal to the base curve, $\xu$ will remain orthogonal to the ruling by virtue of Proposition \ref{Propcoordframeangle}. Since $\xu$ is tangent to the fiber in $I\times_f F$ and $\xv=\partial_t$ is orthogonal to it. It holds
	\begin{equation*}
		\nabla_{\xu} \xv=(\partial_t\log f) \xu,
	\end{equation*}
	so the corresponding Jacobi evolution function reads
	\begin{equation*}
		F(u,v)=g(\xu,\nabla_{\xu} \xv)=(\partial_t\log f)\Vert\xu\Vert^2.
	\end{equation*}
	The solution to the equation $F=0$ which determines the striction curves is given in this case by $f'(v)=0$.
	Hence, there will be as many striction curves as there are values at which the function $f'$ vanishes in $I$. For instance, if we consider $I=\mathbb{R}$ and $(F,g_F)$ isometric to the two-dimensional Euclidean space $(\mathbb{R}^2,\delta)$, for  the choice $f(t)=\sin{t}$, every curve of the form $\alpha_{v_k}(u)=(\pi/2+k\pi,\cos{u},\sin{u})$ with $k\in \mathbb{Z}$ is a striction line in the ruled surface
	determined by it and $\xv=\partial_t$.
	
\vspace{4mm}

\noindent {\bf Acknowledgements}

\noindent The authors are indebted with prof. L.C.B da Silva for his careful reading of the manuscript and for his useful remarks, in particular, about Proposition \ref{ship} above.

\end{document}